\documentclass[a4paper,10pt]{article}
\usepackage{latexsym,amssymb,amsfonts,amsmath,amsthm,amstext,color,graphicx,times}
\usepackage[mathscr]{euscript}
\usepackage{indentfirst}
\usepackage{cite}
\usepackage{mathrsfs}
\usepackage{subcaption}
\usepackage{float}
\newcommand{\R}{\mathbb{R}}

\newcommand{\N}{\mathbb{N}}

\usepackage{wrapfig}
\usepackage[all]{xy}
\usepackage{tikz}
\usepackage{hyperref}
\usepackage{animate}
\usetikzlibrary{patterns}

\theoremstyle{plain}
\newtheorem{theorem}{Theorem}[section]
\newtheorem{lemma}[theorem]{Lemma}
\newtheorem{proposition}[theorem]{Proposition}
\newtheorem{corollary}[theorem]{Corollary}

\theoremstyle{definition}
\newtheorem{definition}[theorem]{Definition}
\newtheorem{example}[theorem]{Example}

\theoremstyle{remark}
\newtheorem{remark}[theorem]{Remark}

\DeclareMathOperator{\gra}{graph}
\DeclareMathOperator{\SVIP}{SVIP}
\DeclareMathOperator{\MVIP}{MVIP}

\newcommand{\tos}{\rightrightarrows} 

\DeclareMathOperator*{\conv}{conv}

\newcommand{\inner}[2]{\langle #1,#2 \rangle}
\title{Existence and uniqueness of maximal elements for preference relations: Variational  approach}

\author{ 
O. Bueno
\thanks{Universidad del Pac\'ifico.  Lima, Per\'u. Email: \texttt{\{o.buenotangoa, cotrina\_je,garcia\_yv\}@up.edu.pe}} 
\and J. Cotrina\footnotemark[2] \and Y. Garc\'ia \footnotemark[3]
}

\begin{document}

\maketitle

\begin{abstract}

In this work, we reformulate the problem of existence of maximal elements for preference relations as a variational inequality problem in the sense of Stampacchia. Similarly, we establish the uniqueness of maximal elements using a variational inequality problem in the sense of Minty. In both of these approaches, we use the normal cone operator to find existence and uniqueness results, under mild assumptions.
In addition, we provide an algorithm for finding such maximal elements, which is inspired by the steepest descent method for minimization. Under certain conditions, we prove that the sequence generated by this algorithm converges to a maximal element.

\bigskip


\bigskip


\end{abstract}
\section{Introduction}

The theory of preference relations is one of the main tools in the study of consumer demand.
A preference relation is described by means of a binary relation, which is traditionally derived from a utility function when the relation satisfies some properties (see for instance \cite{DEBREU}). In that sense, finding a maximal element is equivalent to solving the maximization problem of the associated utility function. However, it is well-known that there are preference relations that are not derived from a utility function, for instance, the lexicographic ordering~\cite{Mas-Colell}. 
We can find in the literature a lot of works concerning the existence of maximal elements for preference relations, which are not necessarily transitive nor complete, see for instance~\cite{CAMPBELL1990,Mehta1984,MILASI2019,MILASI2021,Rader1978,Shafer1974} and the references therein. It is interesting to note that many of these results are consequences of Browder's theorem~\cite[Theorem~1]{Browder1968}.

On the other hand, variational inequalities play an important role in the study of optimization problems, Nash games, saddle problems, among others, see~\cite{Facchinei-book}.  Recently, Milasi \emph{et al.}~\cite{MILASI2019} reformulated the competitive economic equilibrium problem governed by preference relations as a quasi-variational inequality problem. Later on, Milasi and Scopelliti~\cite{MILASI2021} reformulated the problem of finding maximal elements for preference relations as a suitable variational inequality problem. However, we discovered some inconsistencies in these works, see Remarks~\ref{Milasi-e} and~\ref{Milasi-e2}. 

In this work, we aim to study the existence and uniqueness of maximal elements by considering the Stampacchia and Minty variational inequality. 
Very recently, Donato and Villanacci~\cite{donato_variational_2023} also addressed the existence of maximal elements in a similar way. However, our existence result is not a consequence of their results, see Remark~\ref{rem:new3.15}.
In addition, we present an algorithm to obtain these elements, whenever the preference relation is defined on the whole space. This algorithm is inspired by the classical steepest descent method for minimization of quasiconvex functions~\cite{dacruzetal2011}, but instead of the gradient, we use the so-called \emph{Plastria-like normal cone} of the strict upper contour.

In Section~\ref{sec:prelim}, we present definitions and notations of preference relations. Section~\ref{sec:maximal} deals with the existence and uniqueness of maximal elements by considering the variational reformulation. Finally, in Section~\ref{sec:algo} we provide an optimization algorithm for preference relations that satisfy certain conditions.

\section{Preliminaries}\label{sec:prelim}
Consider $X$ a non-empty set and $\succeq$ a binary relation on $X$. 
For each $x\in X$ we  consider the sets
\[
U(x):=\{y\in X\::\: y\succeq x\},\quad L(x):=\{y\in X\::\: x\succeq y\},
\]
\[
U^s(x):=\{y\in X\::\: y\succ x\}\text{ and }L^s(x):=\{y\in X\::\: x\succ y\},
\]
where $\succ $ is the asymmetric part of $\succeq$, that is, $x\succ y$ means $x\succeq y$ but not $y\succeq x$. 
The sets $U(x)$ and $U^s(x)$ (resp. $L(x)$ and $L^s(x)$) are called the \emph{upper} and \emph{strictly upper} (resp. \emph{lower}) \emph{contour set}.

We recall that a binary relation $\succeq$ \cite{Mas-Colell} is said to be:
\begin{itemize}
	\item \emph{Complete} if, for any $x,y\in X$, either $x\succeq y$ or $y\succeq x$.
	\item \emph{Transitive} if, for any $x,y,z\in X$, the following implication holds
	\[
	(x\succeq y~\wedge~y\succeq z)\Rightarrow x\succeq z.
	\]
	\item \emph{Reflexive} if, for any $x\in X$, we have $x\succeq x$.
	\item \emph{Rational} if, it is complete and transitive.
\end{itemize}

\begin{definition}
	Let $m\in\N$. A relation $\succeq$ on $X$ has the $m$-\emph{FIP} if, for any $x_1,x_2,\ldots,x_m\in X$, there exists $x\in X$ such that $x\succeq x_i$, for all $i=1,2,\ldots,m$. If $\succeq$ has the $m$-FIP, for all $m\in \N$, then it has the \emph{finite intersection property}.
\end{definition}

A few remarks are needed.
\begin{remark}
	\begin{enumerate}
		\item The relation $\succeq$ has the finite intersection property if, and only if,  the family of sets $\{U(x)\}_{x\in X}$ has the finite intersection property.
		\item If $\succeq$ is complete, then it has the $2$-FIP. The converse is not true in general, see Example \ref{2-fip-notcomplete}.
	\end{enumerate}
\end{remark}

\begin{proposition}\label{trans-fip}
	Let $X$ be a non-empty set and $\succeq$ be a relation on $X$.
	If $\succeq$ is transitive and has the $2$-FIP, then $\succeq$ has the finite intersection property.
\end{proposition}
\begin{proof}
	We prove this by induction. For $m=1,2$ it follows from the $2$-FIP of $\succeq$.
	
	Assume now that the family of sets $\{U(x)\}_{x\in X}$ satisfies the finite intersection property for $n-1$ elements of $X$ and consider $x_1,\dots,x_n\in X$.
	For $x_1,x_2,\dots,x_{n-1}$ there exists $x\in X$ such that $x\succeq x_i$, for all $i\in\{1,2,\dots,n-1\}$. Thus, again by the $2$-FIP of $\succeq$ there exists $z\in X$ such that $z\succeq x_n$ and $z\succeq x$. By transitivity of $\succeq$, it follows that 
	$z\succeq x_i$ for all $i\in\{1,2,\dots,n\}$. \qed
\end{proof}

In particular, since completeness implies the 2-FIP, we obtain the following corollary.
\begin{corollary}
	Let $\succeq$ be a rational relation on a non-empty set $X$. Then $\succeq$ has the finite intersection property.
\end{corollary}

We now present some examples. Example~\ref{2-fip-notcomplete} shows that neither completeness nor transitivity of $\succeq$ are necessary to guarantee the finite intersection property. On the other hand, Example~\ref{ex:ex23} shows that completeness, 2-FIP or transitivity, separately, are not sufficient to obtain the finite intersection property.
\begin{example}\label{2-fip-notcomplete}
	Consider $X=[0,4]$ and define the relation $\succeq$ on $X$ as
	\[
	x\succeq y \text{ if, and only if, } \dfrac{y}{2}+2\leq x\leq 4 \text{ and }(x,y)\neq (7/2,2).
	\]
	Clearly $4\in\displaystyle\bigcap_{x\in X}U(x)$, hence $\succeq$ has the finite intersection property. However, it is straightforward to verify that $\succeq$ is neither reflexive, nor complete. It is not transitive either, because $7/2\succeq 3$ and $3\succeq 2$ but $7/2\not\succeq 2$.
\end{example}
\begin{example}\label{ex:ex23}
	Consider $X=[0,1]$ and, for $j\in\{a,b\}$, define the relation $\succeq_j$ as 
	\[
	x\succeq_{j} y\text{ if, and only if, }(x,y)\text{ belongs to the graph in Figure~\ref{FIP-transitivity}, item~$j$}.
	\]
	Denote with $U_j(x)$, $j\in\{a,b\}$, the respective upper contour set of $\succeq_j$.
	\begin{figure}[h!]
		\centering
		\begin{tabular}{cc}
			\begin{tikzpicture}[>=latex,scale=1.75]
				\fill[gray!30](0,1)--(0,0)--(1,1)--(0,1);
				\draw[<->](0,1.5)--(0,0)--(1.5,0);
				\draw[line width=0.8](0,1)--(0,0)--(1,1)--(0,1);
				\draw(0,1)node[left]{$1$};
				\draw(1,0)node[below]{$1$};
				\draw[dotted](0,1)--(1,1)--(1,0);
				\draw[fill=white](0,0.5)circle(1.25pt)node[left]{$1/2$};
				\draw[fill=white](0.5,1)circle(1.25pt);
				\fill(1,0.5)circle(1.25pt);
				\fill(0.5,0)circle(1.25pt)node[below]{$1/2$};
			\end{tikzpicture}
			&
			\begin{tikzpicture}[>=latex,scale=1.75]
				\fill[gray!30] (0,0.5) -- (0,0) -- (0.5,0.5) -- cycle;
				\fill[gray!30] (1,0.5) -- (0.5,0.5) -- (1,1) -- cycle;
				\draw[<->](0,1.5)--(0,0)--(1.5,0);
				\draw[line width=0.8](0,0.5)--(0,0)--(1,1)--(1,0.5) -- cycle;
				\draw(0,1)node[left]{$1$};
				\draw(1,0)node[below]{$1$};
				\draw[dotted](0,1)--(1,1)--(1,0);
				\draw[dotted] (0.5,0.5)--(0.5,0) node[below]{$1/2$};
				\node[left] at (0,0.5) {$1/2$};
			\end{tikzpicture}\\
			a) Complete relation & b) Transitive relation
		\end{tabular}
		\caption{}\label{FIP-transitivity}
	\end{figure}
	Note that $\succeq_a$ is complete (hence, it has the 2-FIP). However it is not transitive, since $0\succeq_a 1\succeq_a 1/2$ and $0 \not\succeq_a 1/2$, and it does not have the finite intersection property, because 
	\[
	U_a(0)\cap U_a(1/2)\cap U_a(1) =\{0,1/2\}\cap (]0,1/2]\cup\{1\})\cap ([0,1]\setminus\{1/2\})=\emptyset.
	\]
	On the other hand, it is straightforward to verify that $\succeq_b$ is transitive. Moreover, since $U_b(0)\cap U_b(1)=\emptyset$, $\succeq_b$ does not have the 2-FIP, hence it is not complete nor it has the finite intersection property.
\end{example}

We will now recall some definitions dealing with convexity, which will play an important role later.  Let $X$ be a convex set in a vector space, the relation $\succeq$ is called \emph{convex} (resp. \emph{convex$^s$}), if  $U(x)$ (resp. $U^s(x)$) is convex, for all $x\in X$. These two notions coincide whenever the relation $\succeq$ is rational, as it was established without proof in~\cite[Proposition 2.2, part $(i)$]{MILASI2019}. We present a proof of this fact, to make our work self-contained.
\begin{proposition}\label{le0}
	Assume that $X$ is convex and the relation $\succeq$ is rational. Then $\succeq$ is convex if, and only if, it convex$^s$.
\end{proposition} 
\begin{proof}
	If $U^s(x)=\emptyset$, there is nothing to prove. We now consider $U^s(x)\neq\emptyset$. Let $y,z\in U^s(x)$ and $t\in [0,1]$.  Since $\succeq$ is complete, without loss of generality we can assume $y\succeq z$. Thus, $y,z\in U(z)$ which is a convex set. This implies $ty+(1-t)z\succeq z$ and by transitivity we have $ty+(1-t)z\succ x$.
	
	Reciprocally, consider $y,z\in U(x)$ and $t\in [0,1]$. If there exists $t_0\in]0,1[$ such that 
	$t_0y+(1-t_0)z\notin U(x)$ then $x\succ t_0y+(1-t_0)z$, due to $\succeq$ being complete. Now, by transitivity of $\succeq$, we have that $y,z\in U^s(t_0y+(1-t_0)z)$. However, since $U^s(t_0y+(1-t_0)z)$ is convex we get a contradiction. \qed
\end{proof}

The following examples show that the previous result is not true in general when we drop the rationality.  

\begin{example}
	Consider $X=[0,1]$ and $\succeq$ defined as
	\[
	x\succeq y\text{ if and only if }x=0\text{ or }y=x.
	\]
	It is easy to see that $U(x)=\{0,x\}$, for any $x\in X$. Thus, $\succeq$ is not convex. However,
	\[
	U^s(x)=\begin{cases}
		\emptyset,&x=0,\\
		\{0\},&\text{otherwise},
	\end{cases}
	\]
	which is a convex set, for all $x\in X$. Hence, it is convex$^s$. 
\end{example}

\begin{example}
	Consider $X=[0,1]$ and $\succeq$ defined as
	\[
	x\succeq y\text{ if, and only if, }y=0\text{ or }\begin{cases}
		x=y,&y\in ]0,1]\setminus\{1/2\},\\
		x=0,&y=1/2.
	\end{cases}
	\]
	It is not difficult to see that 
	\[
	U(x)=\begin{cases}
		[0,1],&x=0,\\
		\{x\},&x\in]0,1]\setminus\{0,1/2\},\\
		\{0\},&x=1/2,
	\end{cases}\text{ and }
	U^s(x)=\begin{cases}
		\emptyset,&x\in]0,1],\\
		]0,1/2[\cup]1/2,1],&x=0.
	\end{cases}
	\]
	Thus,  $\succeq$ is convex but it is not convex$^s$.
\end{example}

We finish this section with some topological definitions for binary relations and correspondences. 
Given $X$ a topological space, a relation $\succeq$ on $X$ is said to be \emph{upper} (resp. \emph{lower}) \emph{semicontinuous}, if the set $L^s(x)$ (resp. $U^s(x)$) is open, for all $x\in X$. Moreover, $\succeq$ is \emph{continuous}, if it is upper and lower semicontinuous.
Clearly, if $\succeq$ is complete; then it is upper (resp. lower) semicontinuous if, and only if, $U(x)$ (resp. $L(x)$) is closed, for all $x\in X$.

Let $V, W$ be nonempty sets. A \emph{correspondence} $T:V\tos W$ is an application from $V$ into $\mathcal{P}(W)$, that is, for $v\in V$, $T(v)\subset W$. 
We can see that the  upper contour $U$ and the strict upper contour $U^s$ as examples of correspondences from $X$ to itself.

Recall that a correspondence $T:\R^n\tos \R^m$ is 
\begin{itemize}
	\item \emph{lower hemicontinuous} at $x_0\in \R^n$ when, for any sequence $(x_k)_{k\in\N}$ converging to $x_0$ and any element $y_0$ of $T(x_0)$, there exists a sequence $(y_k)_{k\in\N}$ converging to $y_0$ such that $y_k\in T(x_k)$, for any $k\in\N$.
	\item \emph{upper hemicontinuous} at $x_0\in \R^n$ when, for any neighbourhood $\mathcal{W}$ of $T(x_0)$, there exists a neighbourhood $\mathcal{V}$ of $x_0$ such that $T(x)\subset \mathcal{W}$, for all $x\in \mathcal{V}$;
	\item \emph{lower} (respectively \emph{upper}) \emph{hemicontinuous} when it is lower (resp. upper) hemicontinuous at every $x_0\in \R^n$;
	\item \emph{closed}, if the set $\gra(T):=\{(x,y)\in \R^n\times \R^m\::\: y\in T(x)\}$ is a closed subset of $\R^n\times \R^n$.
\end{itemize} 
It is well known that if $T(\R^n)$ is bounded and $T$ is closed, then it is upper hemicontinuous.  
We also say that $T$ has \emph{open fibres} when the set $T^{-1}(y):=\{x\in \R^n\::\: y\in T(x)\}$ is open, for all $y\in \R^m$. It is easy to verify that if $T$ has open fibres, then it is lower hemicontinuous. 
In particular, if the binary relation $\succeq$ on $\R^n$ is upper (resp. lower) semicontinuous then $U^s:\R^n\tos\R^n$ (resp. $L^s$) has open fibres, hence it is lower (resp. upper) hemicontinuous.

\section{A Variational Approach for Maximal Elements}\label{sec:maximal}
This section is devoted to studying the existence and uniqueness of maximal elements for binary relations. First, we will introduce a normal cone correspondence, related to the strict upper contours of a binary relation, and establish some of its properties. Next, we will deal with the variational reformulation and the existence of maximal elements. Finally, using the Minty variational inequality, we will study the uniqueness of such maximal elements.

Let $\succeq$ be a binary relation on $\R^n$ with asymmetric part $\succ$, let $X$ be a non-empty subset of $\R^n$ and let $\hat x\in X$. Then
\begin{itemize}
	\item $\hat x$ is a \emph{maximum} of $X$, if $\hat{x}\succeq y$, for all $y\in X$;
	\item $\hat x$ is a \emph{maximal element} of $X$, if there is not $y\in X$ such that $y\succ \hat{x}$.
\end{itemize}
We denote by $\mathscr{M}_{\succeq}(X)$ and $\mathscr{ME}_{\succeq}(X)$ the set of maxima and maximal elements of $X$, respectively. It is not difficult to see that $\mathscr{M}_{\succeq}( X)\subset \mathscr{ME}_{\succeq}(X)$, and the equality holds under completeness of $\succeq$. Furthermore, it is clear that
\[
\mathscr{M}_{\succeq}(X)=\bigcap_{x\in X}U(x)\cap X\text{ and }\mathscr{ME}_{\succeq}(X)=\{x\in X\::\:U^s(x)\cap X=\emptyset\}.
\]

\subsection{Variational Inequalities and Normal Cones}
Consider $X$ a subset of $\R^n$ and $T:\R^n\tos \R^n$ a set-valued map. A vector $\hat{x}\in X$ is said to be a solution of the
\begin{itemize}
	
	\item \emph{Stampacchia variational inequality problem} if there exists $\hat{x}^*\in T(\hat{x})$, such that 
	\[
	\langle \hat{x}^*,y-\hat{x}\rangle\geq0,\quad\forall\,y\in X.
	\]
	\item \emph{Minty variational inequality problem}, if
	\[
	\langle y^*,\hat{x}-y\rangle\leq0,\quad\forall\,y\in X,\,\forall\,y^*\in T(y).
	\]
\end{itemize}
The solution sets of the Stampacchia and Minty variational inequality problems will be denoted by $\SVIP(T,X)$  and $\MVIP(T,X)$, respectively.

Given  a subset $A$ of  $\R^n$ and $x\in \R^n$, the \emph{normal cone} of $A$ at $x$ is the set 
\[
\mathscr{N}_A(x):=\begin{cases}
	\{x^*\in \R^n\::\:\langle x^*,y-x\rangle\leq0,~\forall y\in A\},&A\neq\emptyset,\\
	\R^n,&A=\emptyset.
\end{cases}
\]
It is usual in the literature to consider the above definition whenever $A$ convex and closed, and $x\in A$. However, we will not consider such conditions in this work. See Figure~\ref{normal-cone} for a geometric interpretation. Note that $\mathscr{N}_A(x)$ is always non-empty, because $0\in \mathscr{N}_A(x)$. Moreover, it is a closed convex cone of $\R^n$.

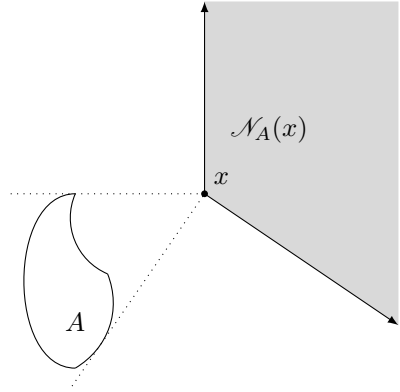
\begin{figure}[h]
	\centering
	\begin{tikzpicture}[scale=0.85,>=latex]
		\fill[gray!30](1,4)--(1,1)--(4,-1.05)--(4,4)--cycle;
		\draw[dotted](-2,1)--(1,1);
		\draw[dotted](-1.05,-2)--(1,1);
		\draw(-1,1) to [bend right=45] (-0.5,-0.25);
		\draw(-0.5,-0.25) to [bend left=40] (-1,-1.72);
		\draw(-1,1) to [bend right=90] (-1,-1.72);
		\draw(-1,-1)node[]{$A$};
		\fill[](1,1)circle(1.5pt)node[ above right]{$x$};
		\draw[->](1,1)--(1,4);
		\draw[->](1,1)--(4,-1.05);
		\draw(2,2)node[]{$\mathscr{N}_A(x)$};
	\end{tikzpicture}
	\caption{The set $\mathscr{N}_A(x)$}\label{normal-cone}
\end{figure}

Let $\succeq$ be a binary relation on $\R^n$.  The \emph{normal cone correspondence} $N:\R^n\tos \R^n$, associated to $\succeq$, is defined as
\begin{equation}6\label{eq:normalcone}
	N(x):=\mathscr{N}_{U^s(x)}(x),
\end{equation}
for all $x\in \R^n$.
\begin{lemma}
	Let $\succeq$ be a complete relation  the following equivalence holds:
	\[
	x^*\in N(x)~\Leftrightarrow~\left(\forall y\in \R^n,~\langle x^*,y-x\rangle>0~\Rightarrow~x\succeq y\right).
	\]
\end{lemma}
In the following example, we explicitly calculate the normal cone correspondence for a certain preference relation.
\begin{example}\label{exN}
	Consider the relation $\succeq$ on $\R$  defined as
	\[
	x\succeq y\text{ if, and only if, }y=x\text{ or }y=1.
	\]
	It is clear that
	\[
	U^s(x)=\begin{cases}
		\emptyset,&x\neq 1,\\
		\R\setminus\{1\},&x=1,
	\end{cases}\text{ and }N(x)=\begin{cases}
		\R,&x\neq 1,\\
		\{0\},&x=1.
	\end{cases}
	\]
\end{example}

\begin{remark}\label{Milasi-e}
	Recently, Milasi \emph{et al.} \cite{MILASI2019},  considered $\succeq$ to be a convex$^s$ and non-satiated relation on $X\subset\R^n$ and defined $M_1:X\tos\R^n$ as
	$M_1(x):=\mathscr{N}_{U^s(x)}(x)$. Clearly, if $X=\R^n$ then $M_1$ coincides with $N$. 
	
	Now, Proposition 2.4 in \cite{MILASI2019} establishes that for all $x^*\in M_1(x)\setminus\{0\}$
	\[
	\langle x^*,y-x\rangle<0,\text{ for all }y\in U^s(x),
	\]
	provided that $\succeq$ is lower semicontinuous, convex$^s$ and with no maximal elements (this property is called \emph{non-satiated} in~\cite{MILASI2019}). However, this result as it is stated is not true. Consider for instance $X=\R\times\{0\}\subset\R^2$ and the relation $\succeq$ on $X$ defined as
	\[
	(x,0)\succeq (y,0) \text{ if, and only if, }x\geq y.
	\]
	It is not difficult to see that $\succeq$ is continuous, convex$^s$ and without maximal elements on $X$. Moreover, for each $(x,0)\in X$ we have
	\[
	U^s(x,0)=]x,+\infty[\times\{0\}\text{ and }M_1(x,0)=\{(x^*,y^*)\in\R^2\::\:x^*\leq 0\}.
	\]
	Since $(0,1)\in M_1(x,0)\setminus\{(0,0)\}$ we obtain that $\langle (0,1), (y,0)-(x,0)\rangle=0$, for all $(y,0)\in U^s(x,0)$.  
\end{remark}

The following result can be found as Proposition~24 in~\cite{donato_variational_2023}. However, our definition of $N$ allows to drop the condition of $U^s(x)$ being non-empty.
\begin{lemma}\label{le34}
	Let  $\succeq$ be a binary relation on $\R^n$ and $x\in\R^n$. If $U^s(x)$ is a convex set, then $N(x)\setminus\{0\}\neq\emptyset$. 
\end{lemma}	

The following result is a generalization of Proposition 2.3 in \cite{MILASI2019}.
\begin{proposition}\label{prop1}
	Let  $\succeq$ be a  relation on $\R^n$. If $\succeq$ is upper semicontinuous, then the map $N$ is closed.
\end{proposition}
\begin{proof}
	Lower semicontinuity of $U^s$ is consequence of the upper semicontinuity of $\succeq$. Hence the proposition follows from Proposition~25 in~\cite{donato_variational_2023}. \qed
\end{proof}

Given a relation $\succeq$ on $\R^n$ we define 
\[
N^*(x):=\{x^*\in\R^n\::\:\langle x^*,y-x\rangle<0,\,\forall\,dy\in U^s(x)\},
\]
whenever $x\notin \mathscr{ME}_{\succeq}(\R^n)$, and $N^*(x):=\R^n$, otherwise.  It is important to note that $N^*(x)$ is a cone and $N^*(x)\subset N(x)$, for all $x\in X$.

\begin{proposition}
	Let $\succeq$ be a rational relation on $\R^n$.
	If $N^*(x)\neq\emptyset$ for all $x\in \R^n$; then $\succeq$ is convex.
\end{proposition}
\begin{proof}
	Let $x\in\R^n$. If $U(x)=\emptyset$, there is nothing to prove. Now assume that $U(x)\neq\emptyset$, and take any $y\notin U(x)$. Choosing some $y^*\in N^*(y)$,
	define the set
	\[
	H^-(y)=\{z\in \R^n\::\:\langle y^*,z-y\rangle<0\}.
	\] 
	Clearly, $H^-(y)$ is convex and $y\notin H^-(y)$. Since $y$ was taken arbitrarily, we have the inclusion
	\[
	\bigcap_{y\notin U(x)} H^-(y)\subset U(x).
	\] 
	On the other hand, for $y\notin U(x)$, since $\succeq$ is complete, $x\succ y$. We now claim that $U(x)\subset H^-(y)$. Indeed, for all $z\in U(x)$ we have $z\succeq x\succ y$, hence $z\succ y$, by transitivity of $\succeq$. 
	Thus $\langle y^*,z-y\rangle<0$, because $y^*\in N^*(y)$, which implies $z\in H^-(y)$. Therefore
	\[
	U(x)= \bigcap_{y\notin U(x)} H^-(y).
	\] 
	The proposition follows. \qed
\end{proof}

The following proposition shows the relation between the operators $N$ and $N^*$.
\begin{proposition}\label{pro:37}
	Let  $\succeq$ be a lower semicontinuous relation on $\R^n$. Then
	\begin{enumerate}
		\item $N(x)=N^*(x)\cup\{0\}$, for all $x\in \R^n$.
		\item If, in addition, $\succeq$ is convex$^s$, then $N^*$ is non-empty valued.
	\end{enumerate}
\end{proposition}
\begin{proof}
\begin{enumerate}
	\item The lower semicontinuity of $\succeq$ implies that $U^s(x)$ is open. Since $N(x)$ is the polar cone of $U^s(x)\setminus\{x\}$, this item now follows using~\cite[Exercise~6.22]{RW}. 
	\item It follows from item~{\it 1.} and Lemma~\ref{le34}. \qed
\end{enumerate}
\end{proof}
The following example shows that it is not possible to drop the lower semicontinuity in the first part of the previous result.
\begin{example}
	Consider the relation $\succeq$ on $\R^2$ defined as
	\[
	(x,y)\succeq (a,b)\text{ if, and only if, }x\geq a \text{ and }y=b=0.
	\]
	It is not difficult to verify that
	\[
	U^s(x,y)=\begin{cases}
		\emptyset,&y\neq0 ,\\
		]x,+\infty[\times\{0\},&y=0.
	\end{cases}
	\]
	Thus, the relation $\succeq$ is not lower semicontinuous. Additionally, we can see that
	\[
	N(x,y)=\begin{cases}
		\R^2,&y\neq0,\\
		\{(x^*,y^*)\in\R^2\::\:x^*\leq0\},&y=0,
	\end{cases}
	\]
	and 
	\[
	N^*(x,y)=\begin{cases}
		\R^2,&y\neq0,\\
		\{(x^*,y^*)\in\R^2\::\:x^*<0\},&y=0.
	\end{cases}
	\]
	Thus,  $N(x,0)\neq N^*(x,0)\cup\{(0,0)\}$.
\end{example}

\subsection{Reformulation and Existence Results}

The following correspondence will allow us to reformulate the problem of finding maximal elements as a Stampacchia variational inequality problem.
Consider  the correspondence $T:\R^n\tos \R^n$ defined as
\begin{equation}\label{opeT}
	T(x):=\conv(N(x)\cap S[0,1]),
\end{equation}
where $N$ is the correspondence associated to $\succeq$, defined as in~\eqref{eq:normalcone}, and $S[0,1]$ denotes the unit sphere of $\R^n$. For instance, if $\succeq$ and $N$ are defined as in Example~\ref{exN}, then the correspondence $T$ associated to $N$ satisfies $T(x)=[-1,1]$, when $x\neq 1$, and $T(1)=\emptyset$.
 
The following proposition establishes a relation between $\succeq$ and $T$.
\begin{proposition}\label{two-parts}
	The following implications hold:
	\begin{enumerate}
		\item If $\succeq$ is upper semicontinuous on $\R^n$, then $T$ is closed. In particular $T$ is upper hemicontinuous.
		\item If $\succeq$ is convex$^s$ and lower semicontinuous on $\R^n$, then $T$ is non-empty valued.
	\end{enumerate}
\end{proposition}
\begin{proof}
\begin{enumerate}
\item Consider the correspondence $R:\R^n\tos \R^n$ such that $\gra(R)=\gra(N)\cap \big(\R^n\times S[0,1]\big)$. By Proposition \ref{prop1}, we deduce that $R$ is closed. Moreover, it is upper semicontinuous due to the closed graph theorem. Since
$T(x)=\conv(R(x))$, for all $x\in X$, the result follows from Theorem~17.35 in \cite{aliprantis06}.
\item It follows from Proposition~\ref{pro:37}, item~{\it 2}. \qed
\end{enumerate} 
\end{proof}

\begin{remark}\label{Milasi-e2} 
	Milasi and Scopelliti~\cite{MILASI2021} defined the map $M_2:\R^n\tos\R^n$ as
	\[
	M_2(x):=\begin{cases}
		\mathscr{N}_{U^s(x)}(x),&x\in X,\\
		\emptyset,&x\notin X.
	\end{cases}
	\]
	The authors also introduced a correspondence similar to $T$ given in~\eqref{opeT}. More precisely, they defined the map $G:X\tos\R^n$ as
	\[
	G(x):=\begin{cases}
		\conv(M_2(x)\cap S[0,1]),&U^s(x)\neq\emptyset,\\
		\overline{B}(0,1),&U^s(x)=\emptyset, 
	\end{cases}
	\]
	where $\overline{B}(0,1)$ is closed unit ball of $\R^n$. 
	It is clear that $T$ and $G$ coincide when $X=\R^n$.
	
	Now, part b) of Theorem 3 in~\cite{MILASI2021} establishes that any solution of the variational inequality problem (in the sense of Stampacchia) associated to $G$ and $K\subset X$ is a maximal element of $\succeq$ on $K$. However, this is not true as it is stated. Consider for instance the relation $\succeq$ given in Remark~\ref{Milasi-e}, and $K=[0,1]\times\{0\}\subset X$. The map $G:X\tos\R^2$ is given by
	\[
	G(x,0)=\{(x^*,y^*)\in\R^2\::\:x^*\leq0\text{ and }(x^*)^2+(y^*)^2\leq1\}.
	\]
	Since $(0,0)\in G(x,0)$ for all $x\in \R$, we deduce that the solution set of the variational inequality problem, associated to $G$ and $K$, coincides with $K$. However, the unique maximal element of $K$ is $(1,0)$.
\end{remark}

Motivated by the previous remark we establish the following result. 
\begin{proposition}\label{SVI-ME}
	Assume that $X$ is a non-empty subset of $\R^n$ and let $\succeq$ be lower semicontinuous relation on $\R^n$. Then $\SVIP(T,X)\subset \mathscr{ME}_{\succeq}(X)$.
\end{proposition}
\begin{proof}
	Let $\hat{x}$ be an element of $\SVIP(T,X)$. There exists $\hat{x}^*\in T(\hat{x})$ satisfying
	\begin{equation}\label{SVI-2}
		\langle \hat{x}^*,y-\hat{x}\rangle\geq0,\text{ for all }y\in X.
	\end{equation}
	By definition of $T$, there are $x_1^*,x_2^*,\dots,x_m^*\in N(\hat{x})\cap S[0,1]$ and $t_1,t_2,\dots,t_m\in [0,1]$ such that $\hat{x}^*=\sum_{i=1}^mt_ix_i^*$ and $\sum_{i=1}^m t_i=1$.
	Assume that $\hat{x}$ is not a maximal element of $\succeq$ on $X$. Then there exists $y\in X$ such that $y\succ \hat{x}$, i.e. $y\in U^s(\hat{x})$. Using inequality~\eqref{SVI-2}, we deduce $\langle x_i^*,y-\hat{x}\rangle\geq 0$ for some $i$. On the other hand, since $x_i^*\in N(\hat x)\setminus\{0\}$ and Proposition~\ref{pro:37}, item~{\it 1.}, $\langle x_i^*,y-\hat{x}\rangle<0$, a contradiction. \qed
\end{proof}

We now are able to present the main result of this subsection.
\begin{theorem}\label{existence-vip-max}
	Let $X$ be a convex, compact and non-empty subset of $\R^n$, and let $\succeq$ be a binary relation on $\R^n$. If $\succeq$ is continuous and convex$^s$; then there exists at least a maximal element for $\succeq$ on $X$.
\end{theorem}
\begin{proof}
	From Proposition \ref{two-parts}, the correspondence $T$ is upper hemicontinuous with convex, compact and non-empty values. By Theorem 9.9 in \cite{JP-Aubin-1998}, there exists $\hat{x}\in\SVIP(T,X)$. Therefore, the result follows from Proposition \ref{SVI-ME}. \qed
\end{proof}

\begin{remark}\label{rem:new3.15}
	It is important to note that in Proposition~\ref{SVI-ME} we do not require convexity$^s$ of the relation, contrary to Theorem 28 in \cite{donato_variational_2023}. On the other hand, Theorem~\ref{existence-vip-max} is not a consequence of Theorem~33 in \cite{donato_variational_2023}. Indeed, consider the relation $\succeq$ on $\R$, defined as
	\[
	x\succeq y\mbox{ if, and only if, }x=y=0.
	\]
	Clearly, $\succeq$ is continuous and convex$^s$. Moreover, $T(x)=[-1,1]$, for all $x\in\R$. Hence, for any compact, convex and non-empty subset $X$ of $\R$, by Theorem~\ref{existence-vip-max}, $\mathscr{ME}_{\succeq}(X)\neq\emptyset$. However, we cannot apply Theorem 33 in \cite{donato_variational_2023}, as $\succeq$ allows for empty strict upper contours.
\end{remark}

\subsection{On the Uniqueness of Maximal Elements}
In this subsection we will show a result concerning the uniqueness of maximal elements. Before this, we need some previous results.
The first result is about necessary conditions, while the second provides sufficient conditions. Both are related to the Minty variational inequality problem.

\begin{lemma}\label{lem:3.14}
	Let $X$ be a subset of $\R^n$ and $\succeq$ be a complete relation on $\R^n$. 
	If $\hat{x}\in \mathscr{ME}_{\succeq}(X)$, then 
	\begin{equation}\label{eq:lem314}
		\langle y^*,\hat{x}-y\rangle\leq0,~\text{for all } y\in X\setminus\mathscr{ME}_{\succeq}(X)\text{ and all }y^*\in N(y).
	\end{equation}
\end{lemma}
\begin{proof}
	Since $\succeq$ is complete,  for any $y\in X\setminus\mathscr{ME}_{\succeq}(X)$, we have 
	$\hat{x}\in U^s(y)$. The result follows from the definition of $N(y)$. \qed
\end{proof}

The following example shows that we cannot drop the completeness of Lemma~\ref{lem:3.14}.
\begin{example}\label{ex:315}
	Consider $\succeq$ on $\R$ defined as
	\[
	x\succeq y\text{ if, and only if, }(x,y)=(0,0) \vee (x\geq y~\wedge~x\neq0),
	\]
	and let $X=\R$. 
	It is clear that $U^s(x)=]x,+\infty[$ for all $x\neq 0$, and $U^s(0)=\emptyset$, hence $\mathscr{ME}_{\succeq}(\R)=\{0\}$. Moreover, we can see that $N(x)=]-\infty,0]$ for all $x\neq0$, and $N(0)=\R$. However, $\hat x=0$ does not satisfy inequality~\eqref{eq:lem314}.
\end{example}

\begin{proposition}\label{pro:316}
	Let $\succeq$ be a binary relation on $\R^n$. If the sets $\mathscr{ME}_{\succeq}(\R^n)$ and $\MVIP(N,\R^n)$ are both non-empty, then $\mathscr{ME}_{\succeq}(\R^n)=\MVIP(N,\R^n)=\{\hat x\}$, for some $\hat x\in \R^n$.
\end{proposition}
\begin{proof}
	Let $\hat x\in \MVIP(N,\R^n)$ and $y\in \mathscr{ME}_{\succeq}(\R^n)$ be arbitrary. Since $y$ is maximal, $U^s(y)=\emptyset$ and $N(y)=\R^n$. If $\hat x\neq y$ then, there exists $y^*\in N(y)=\R^n$ such that $\inner{y^*}{\hat x-y}>0$, a contradiction with $\hat x\in \MVIP(N,\R^n)$. Hence $\hat x=y$, for all $\hat x\in \MVIP(N,\R^n)$ and all $y\in \mathscr{ME}_{\succeq}(\R^n)$. This implies the proposition. \qed
\end{proof}

We now present the main result of this subsection.
\begin{theorem}\label{preference-Minty}
	Let $\succeq$ be a complete relation on $\R^n$. If $\mathscr{ME}_{\succeq}(\R^n)\neq \emptyset$, then $\mathscr{ME}_{\succeq}(\R^n)$ is a singleton if, and only if, $\mathscr{ME}_{\succeq}(\R^n)=\MVIP(N,\R^n)$.
\end{theorem}
\begin{proof}
	First assume that $\mathscr{ME}_{\succeq}(\R^n)=\{\hat x\}$, and take $y\in\R^n$ and $y^*\in N(y)$. If $y\neq \hat x$, Lemma~\ref{lem:3.14} implies $\inner{y^*}{\hat x-y}\leq 0$, and, when $y=\hat x$, trivially $\inner{y^*}{\hat x-y}= 0$. Hence $\hat x\in \MVIP(N,\R^n)$ and, by Proposition~\ref{pro:316}, $\mathscr{ME}_{\succeq}(\R^n)=\MVIP(N,\R^n)=\{\hat x\}$ . The converse implication is a direct consequence of Proposition~\ref{pro:316}. \qed
\end{proof}

\begin{remark}
	\begin{enumerate}
		\item Proposition~\ref{pro:316} implies that if either $\mathscr{ME}_{\succeq}(\R)$ or $\MVIP(N,\R)$ have at least two elements, then the other one must be empty.
		\item Example~\ref{ex:315} also implies that the completeness of $\succeq$ cannot be dropped in Theorem~\ref{preference-Minty}. Indeed, for $\succeq$ as in Example~\ref{ex:315}, $\mathscr{ME}_{\succeq}(\R)=\{0\}$ but in this case $\MVIP(N,\R)=\emptyset$.
		\item Existence of maximal elements can be guaranteed by assuming that $\succeq$ is complete, has the finite intersection property, $\bigcap_{x\in X}U(x)=\bigcap_{x\in X}\overline{U(x)}$ and $U(x)$ is bounded for some $x$
	\end{enumerate}
\end{remark}

\section{An Algorithm to Find Maximal Elements}\label{sec:algo}
This section is divided in two parts: the first one deals with the definition and properties of a normal cone operator inspired by the Plastria subdifferential~\cite{plastria_lower_1985}. The second subsection contains an algorithm for finding maximal elements of a binary relation, along with convergence results.

\subsection{The Plastria-like Normal Cone}
Given a relation $\succeq$ on $\R^n$ and a function $f:\R^n\times \R^n \to\R$ we define
\[
N_f(x):=\begin{cases}
	\{x^*\in\R^n\::\:\langle x^*,y-x\rangle\leq  f(x,y),\,\forall\,y\in U^s(x)\},&x\notin \mathscr{ME}_{\succeq}(\R^n),\\
	\R^n,&\text{otherwise}.
\end{cases}
\]
We will call $N_f(x)$ as the \emph{Plastria-like normal cone} associated to the binary relation $\succeq$ with respect to the function $f$.
Clearly, if $f=0$ then the Plastria-like normal cone $N_f$ reduces to the classical normal cone defined in~\eqref{eq:normalcone}. It is also clear that $N_f(x)$ is closed and convex, for all $x\in \R^n$. 
In general, $N_f(x)$ may not be a cone, as it is shown by the following example.
\begin{example}
	Consider $\succeq$ on $\R$ as in Example~\ref{ex:315} and functions $f_1,f_2:\R\times\R\to\R$ defined as
	\[
	f_1(x,y)=y^2-x^2\text{ and }f_2(x,y)=y-x.
	\]
	It is not difficult to show that  {the Plastria-like normal cones of $\succeq$, with respect to $f_1$ and $f_2$, respectively, are}
	\[
	N_{f_1}(x)=\begin{cases}
		]-\infty,2x],&x\neq0,\\
		\R,&x=0,
	\end{cases}\text{ and }N_{f_2}(x)=\begin{cases}
		]-\infty,1],&x\neq0,\\
		\R,&x=0.
	\end{cases}
	\]
\end{example}

The following proposition extends Proposition~25 in \cite{donato_variational_2023}.
\begin{proposition}\label{pro:new}
	Let $\succeq$ be an  upper semicontinuous relation on $\R^n$ and $f:\R^{n}\times\R^n\to\R$ be a function. If $f$ is upper semicontinuous then  {the Plastria-like normal cone} $N_f$ is closed.
\end{proposition}
\begin{proof}
	Let $(x_k,x_k^*)$ be a sequence converging to $(x,x^*)$ such that $x_k\in \R^n$ and $x_k^*\in N_f(x_k)$, for all $k\in\N$. If $U^s(x)=\emptyset$, there is nothing to prove. Now, we consider that $U^s(x)\neq\emptyset$ and take any $y\in U^s(x)$. Since $U^s$ is a lower hemicontinuous correspondence, there exists a sequence $(y_k)$ converging to $y$ such that $y_k\in U^s(x_k)$. Thus,
	\[
	\langle x^*_k,y_k-x_k\rangle\leq f(x_k,y_k).
	\]
	Letting $k$ tend to $\infty$, we obtain $\langle x^*,y-x\rangle\leq f(x,y)$. Since $y$ was arbitrary, we conclude that $x^*\in N_f(x)$. Therefore, $N_f$ is closed. \qed
\end{proof}
Note that Proposition~\ref{pro:new} coincides with Proposition~\ref{prop1}, when $f=0$. 
Also, this proposition holds even if there is no connection between $\succeq$ and $f$.
In the case when $f$ and $\succeq$ are related, we obtain some other properties, listed in the following propositions.

\begin{proposition}\label{pro:newNF1}
	Let $\succeq$ be a relation on $\R^n$ and $f:\R^{n}\times\R^n\to\R$ be a function such that $f(x,y)<0$ if, and only if, $y\in U^s(x)$. The following hold:
	\begin{enumerate}
		\item $0\in N_f(x)$ if, and only if, $x\in\mathscr{ME}_\succeq(\R^n)$.
		\item If $f$ is upper semicontinuous with respect to its first variable, then $\succeq$ is upper semicontinuous.
		\item If $f$ is upper semicontinuous with respect to its first variable, and $(x_0,0)\in\overline{\gra(N_f)}$, then $0\in N_f(x_0)$.
	\end{enumerate}	
\end{proposition}
\begin{proof}
	\begin{enumerate}
		\item Follows directly from the definition of $N_f$.
		\item It is enough to note that upper semicontinuity of $f(\cdot,y)$ implies that the set
		\[
		L^s(y)=\{x\::\: y\succ x\}=\{x\::\: f(x,y)<0\}
		\]
		is open.
		\item If $x_0\in \mathscr{ME}_{\succeq}(\R^n)$ there is nothing to prove. Otherwise, there exists $\hat x\in U^s(x_0)$, which implies $f(x_0,\hat x)<0$ and $x_0\in L^s(\hat x)$. Let $(x_k,x_k^*)\in\gra(N_f)$, such that $x_k\to x_0$ and $x_k^*\to 0$. 
		Now, item~{\it 2} implies that $L^s(\hat x)$ is open, therefore  $x_k\in L^s(\hat x)$,
		for $k$ large enough. Since $x_k^*\in N_f(x_k)$, we deduce
		\[
		\inner{x_k^*}{\hat x-x_k}\leq f(x_k,\hat x).
		\]
		Taking the limit when $k\to\infty$, we obtain $0\leq f(x_0,\hat x)<0$, a contradiction. \qed
	\end{enumerate}
\end{proof}

\begin{proposition}\label{P-f}
	Let $\succeq$ be a relation on $\R^n$ and $f:\R^{n}\times\R^n\to\R$ be a function  such that the following hold:
	\begin{enumerate}
		\item for all $x,y\in \R^n$, if $f(x,z)>f(y,z)$, for some $z\in\R^n$, then $x\succ y$;
		\item there exists $L>0$ such that $|f(x,y)|\leq L\|x-y\|$, for all $x,y\in \R^n$.
	\end{enumerate}
Then, for every $u^*\in N^*(x_0)$, $x_0^*:=\dfrac{Lu^*}{\|u^*\|}\in N_f(x_0)$.
\end{proposition}
\begin{proof}
	If $x_0\in \mathscr{ME}_{\succeq}(\R^n)$, there is nothing to prove. Otherwise, $U^s(x_0)\neq  \emptyset$. Take $u^*\in N^*(x_0)$, that is, 
	\[
	\langle u^*,x-x_0\rangle <0,\text{ for all }x\in U^s(x_0).
	\]
	Now we define $x_0^*=\dfrac{L}{\|u^*\|}u^*$ and consider the hyperplane $H=\{z\in \R^n\::\:\langle u^*,z-x_0\rangle=0\}$. For each $x\in U^s(x_0)$ we consider $x'\in H$ as the projection of $x$ onto $H$. It is not difficult to verify that 
	\[
	\langle \frac{u^*}{\|u^*\|},x-x_0\rangle=-\|x-x'\|.
	\] 
	Since $x'\notin U^s(x_0)$, assumptions {\it 1.} and {\it 2.} imply $f(x_0,x)\geq f(x',x)\geq -L\|x-x'\|$. Hence
	\[
	f(x_0,x)\geq \langle x_0^*,x-x_0\rangle,
	\]
	and, therefore, $x_0^*\in N_f(x_0)$. The proof is complete. \qed
\end{proof}

\begin{remark}\label{rem:plastria}
	In the particular case when $\succeq$ is represented by a utility function $u:\R^n\to\R$ (in the sense that $x\succeq y$ if, and only if, $u(x)\geq u(y)$), we can define
	the function $f_u:\R^n\times\R^n\to\R$, $f_u(x,y):=u(x)-u(y)$. Using $f_u$, the  {Plastria-like normal cone} $N_{f_u}(x)$ coincides with the \emph{Plastria lower subdifferential} of $u$ at $x$~\cite{plastria_lower_1985}. 
	
	On the other hand, Proposition~\ref{pro:newNF1}, item~{\it 1.}, extends Theorem~3.1 in~\cite{plastria_lower_1985} and item~{\it 3.} extends Proposition~7 in~\cite{dacruzetal2011}. 
	
	Note that, using $f_u$ as above, we readily obtain condition~{\it 1.} in Proposition~\ref{P-f}. Moreover, condition~{\it 2.} is equivalent to the function $u$ be Lipschitz continuous on $\R^n$.  In view of this, Proposition~\ref{P-f} extends Theorem~20 in~\cite{censor_algorithms_2006}. 
\end{remark}

\subsection{The Algorithm}
We begin this subsection by recalling the definition of a quasi-Fej\'er monotone sequence.
A sequence $\{x_k\}\subset \R^n$ is called \emph{quasi-Fej\'er monotone} with respect to $M\subset\R^n$ if for every $u\in M$ there exists a sequence $\{\varepsilon_k\}\subset\R_+$ with $\sum \varepsilon_k<\infty$, such that
\[
\|x_{k+1}-u\|^2\leq \|x_k-u\|^2+\varepsilon_k.
\]
We state below an important result concerning quasi-Fej\'er monotone sequences~ \cite[Theorem~5.33]{bauschke_convex_2011}.
\begin{theorem}\label{Bauschke-convex}
	Let $\{x_n\}$  be a sequence in a Hilbert space $H$ and let $C$ be a non-empty subset of $H$ such that $(x_n)_{n\in\N}$ is quasi-Fej\'er  monotone with respect to $C$. Then the following hold:
	\begin{enumerate}
		\item $\{x_n\}$ is bounded.
		\item Suppose that every weak sequential cluster point of $\{x_n\}$ belongs to $C$. Then $\{x_n\}$ converges weakly to a point in $C$.
	\end{enumerate}
\end{theorem}

We are now ready to present an algorithm for finding maximal elements of preference relations defined in $\R^n$.
In the sequel, we will assume the following conditions:
\begin{enumerate}
	\item $\succeq$ is convex$^s$,
	\item there is a function $f:\R^n\times \R^n\to\R$ satisfying
	\begin{enumerate}
		\item $f(x,y)<0$ if, and only if, $y\in U^s(x)$;
		\item $f(x,y)>0$ if, and only if, $x\in U^s(y)$;
		\item there is $L>0$ such that $|f(x,y)|\leq L\|x-y\|$, for all $x, y\in \R^n$;
		\item for all $x,y\in \R^n$, $x\succ y$ if, and only if, $f(x,z)>f(y,z)$, for all $z\in\R^n$, if, and only if,  $f(x,z)>f(y,z)$, for some $z\in\R^n$; and
		\item $f$ is upper semicontinuous with respect to its first variable.
	\end{enumerate}
	\item $\mathscr{ME}_{\succeq}(\R^n)$ is non-empty.
\end{enumerate}
Under these conditions we can write an analog of the classical steepest descent method, as follows:
\begin{itemize}
\item[] {\bf Initialization:} Take $x_1\in \R^n$, arbitrarily.
\item[] {\bf Iteration:} Given $x_k$, calculate the next iterative $x_{k+1}$ from
	\[
	x_{k+1}:=x_k-\theta_kx_k^*,~k\in\mathbb{N},
	\]
	where $\theta_k$ is some positive real number such that
	\[
	\sum \theta_k=\infty\text{ and }\sum \theta_k^2<\infty,
	\]
	and $x_k^*\in N_f(x_k)$ such that $\|x_k^*\|\leq L$ with $L>0$.
\end{itemize}
Notice that item (d) of assumption~{\it 2.} in the algorithm implies assumption~{\it 1.} in Proposition \ref{P-f}. Also, items (b) and (d) of assumption~{\it 2.} imply the transitivity of $\succeq$. Moreover, assumption~{\it 2.} in the algorithm does not imply the completeness of $\succeq$ as we can see in the following example.
\begin{example}\label{ex:new46}
	Consider $\succeq$ defined as in Remark~\ref{rem:new3.15}, that is,  $x\succeq y$ if, and only if, $x=y=0$, which is clearly non-complete. However, choosing $f\equiv 0$ on $\R^2$, $f$ and $\succeq$ trivially satisfy conditions {\it 1.}, {\it 2.} and {\it 3.} in the algorithm.
\end{example}

It is important to notice that if $x_k^*=0$, for some $k$, then by Proposition~\ref{pro:newNF1}, item~{\it 1}, $x_k$ is a maximal element. From now on, we assume that the sequence $\{x_k\}$ generated by the algorithm is infinite, that is $x_k^*\neq 0$, for all $k\in\N$. 
It is clear that $\mathscr{ME}_{\succeq}(\R^n)\subset \displaystyle\bigcap_{k\in\mathbb{N}}U^s(x_k)$.
\begin{theorem}
	The sequence $\{x_k\}$ generated by the Algorithm converges to a point $\hat{x}\in \mathscr{ME}_{\succeq}(\R^n)$.
\end{theorem}
\begin{proof}
	Notice that
	\[
	\|x_{k+1}-x_k\|^2=\theta_k^2\|x_k^*\|^2= L^2\theta_k^2,
	\]
	which implies the series $\sum \|x_{k+1}-x_k\|^2$ is convergent. On the other hand,
	\begin{equation}\label{algo1}
		\|x_{k+1}-u\|^2=\|x_k-u\|^2+\|x_{k+1}-x_k\|^2+2\theta_k\langle x_k^*,u-x_k\rangle
	\end{equation}
	for all $k\in\mathbb{N}$ and all $u\in \R^n$. Taking $\hat{x}\in \mathscr{ME}_{\succeq}(\R^n)$ in the previous equality we have
	\[
	\|x_{k+1}-\hat{x}\|^2\leq \|x_k-\hat{x}\|^2+\|x_{k+1}-x_k\|^2.
	\]
	Thus, the sequence $\{x_k\}$ is quasi-Fejer monotone with respect to $\mathscr{ME}_{\succeq}(\R^n)$ with $\varepsilon_k=\|x_{k+1}-x_k\|^2$. Moreover, from~\eqref{algo1} we obtain
	\[
	\|x_{k+1}-\hat{x}\|^2\leq\|x_k-\hat{x}\|^2+\|x_{k+1}-x_k\|^2+2\theta_kf(x_k,\hat{x}),
	\]
	which in turn implies
	\[
	-2\theta_kf(x_k,\hat{x})\leq \left(\|x_k-\hat{x}\|^2-\|x_{k+1}-\hat{x}\|^2\right)+\|x_{k+1}-x_k\|^2.
	\]
	Since both series $\sum \|x_{k+1}-x_k\|^2$ and $\sum \left(\|x_k-\hat{x}\|^2-\|x_{k+1}-\hat{x}\|^2\right)$ are convergent,  we obtain
	\[
	\sum -2\theta_kf(x_k,\hat{x})<\infty.
	\]
	Hence $\lim f(x_k,\hat{x})=0$. 
	
	Finally, part~{\it 1.} of Theorem~\ref{Bauschke-convex} implies that $\{x_k\}$ is bounded. Without loss of generality, we may assume that $x_k$ converges to some $x_0$, so, due to the upper semicontinuity of $f(\cdot,\hat{x})$,  $f(x_0,\hat{x})\geq 0$. If $x_0$ is not maximal, there exists $z\in X$ such that $z\succ x_0$. Thus, $f(z,\hat{x})>f(x_0,\hat{x})\geq0$, which in turn implies $z\succ \hat{x}$ and we get a contradiction because $\hat{x}$ is a maximal element. The result now follows from part~{\it 2.} of Theorem \ref{Bauschke-convex}. \qed
\end{proof}

\begin{remark}
	Assume that $\succeq$ is represented by a utility function $u:\R^n\to\R$, and consider $f_u$ as in Remark~\ref{rem:plastria}, that is, $f_u(x,y):=u(x)-u(y)$. If $u$ is quasiconcave and Lipschitz, and possesses a maximum on $\R^n$, then $f_u$ satisfy assumptions~{\it 1.}, {\it 2.} and {\it 3.} of our algorithm. In view of this, our algorithm generalizes Algorithm~A1 in~\cite{dacruzetal2011}, as we do not require differentiability of $u$.
\end{remark}

\section{Conclusions}
In this work, we improved some existence results on maximal elements via a variational approach. 
In particular, we compared our results with Proposition~2.4 in~\cite{MILASI2019} and Theorem~3 in~\cite{MILASI2021}. We have identified a flaw in these results, which we described in Remarks~\ref{Milasi-e} and \ref{Milasi-e2}.  Moreover, we pointed out that our main result is not a consequence of Theorems~28 and~33 in~\cite{donato_variational_2023}. 
We also characterized the uniqueness of maximal elements, by studying the solutions of a certain Minty variational inequality problem.

Finally, we adapted the classical steepest descent method, where we use a new kind of normal cone, called Plastria-like normal cone, to replace the usual derivative. In this way, we established an algorithm to find maximal elements. This algorithm extends Algorithm~A1 in~\cite{dacruzetal2011}.

\end{document}